\documentclass[a4paper,12pt]{article}

\usepackage{amssymb,latexsym,amsthm,amsmath,stmaryrd,amsfonts}
\usepackage{mathabx}
\usepackage{graphicx}
\usepackage{color}%
\usepackage{times}
\newcommand{\R}{\mathbb{R}}

\DeclareMathOperator{\dom}{dom}
\DeclareMathOperator{\ran}{ran}
\DeclareMathOperator{\gra}{gra}

\newtheorem{theorem}{Theorem}[section]
\newtheorem{lemma}[theorem]{Lemma}
\newtheorem{proposition}[theorem]{Proposition}
\newtheorem{corollary}[theorem]{Corollary}

\theoremstyle{definition}
\newtheorem{definition}[theorem]{Definition}
\newtheorem{example}[theorem]{Example}

\theoremstyle{remark}
\newtheorem{remark}[theorem]{Remark}

\newcommand{\tos}{\rightrightarrows} 

\newcommand{\ds}{\displaystyle}

\newcommand{\inner}[2]{\langle #1,#2 \rangle}

\usepackage{hyperref}
\title{On maximality of quasimonotone operators}
\author{Orestes Bueno
\thanks{Universidad del Pac\'ifico. Av. Salaverry 2020, Jes\'us Mar\'ia, Lima, Per\'u. Email: \texttt{\{o.buenotangoa, cotrina\_je\}@up.edu.pe}} 
\and John Cotrina\footnotemark[1]}

\DeclareMathOperator{\cone}{cone}
\DeclareMathOperator{\co}{co}
\DeclareMathOperator{\edom}{edom}

\begin{document}

\maketitle

\begin{abstract}
We introduce the notion of quasimonotone polar of a multivalued operator, 
in a similar way as the well-known monotone polar due to Martinez-Legaz and Svaiter.  
We first recover several properties similar to the monotone polar, including a characterization in terms of normal cones. 
Next, we use it to analyze certain aspects of maximal (in the sense of graph inclusion) quasimonotonicity, 
and its relation to the notion of maximal quasimonotonicity introduced by Aussel and Eberhard.
Furthermore, we study the connections between quasimonotonicity and Minty Variational Inequality Problems.

\bigskip


\bigskip

\end{abstract}

\section{Introduction}

The monotone polar of a multivalued operator was introduced in 2005 by Martinez-Legaz and Svaiter~\cite{BSML}. Here, the authors studied the connection between such polar and representability and maximality of monotone operators, along with certain geometrical properties. 

The aim of this work is to extend the notion of monotone polar to the quasimonotone case. 
We show that the \emph{quasimonotone polar} shares the same properties with the usual monotone polar, 
and allows a characterization of quasimonotonicity, in a parallel way than in~\cite{BSML}.  
Moreover, we characterize the images of quasimonotone polar in terms of the normal cone of a certain set.

We also study the property of maximality of this kind of operators. 
We understand maximality in the sense of graph inclusion, which is the same sense as is usually 
addressed for the monotone case.  As expected, we show that the quasimonotone polar is useful to 
characterize maximal quasimonotonicity and also \emph{pre-maximal} quasimonotonicity, that is, 
operators which have an unique maximal quasimonotone extension. Moreover, we show the quasimonotone polar 
behaves well with the \emph{lateral translation} of an operator.

In a recent paper~\cite{Aussel2013-4}, Aussel and Eberhard introduced another notion of maximality 
for quasimonotone operators, which we call \emph{AE-maximality}. We study the relation between the 
graph inclusion maximality and the AE-maximality, proving that there are, in a natural sense, equivalent.  
In addition, we extend some results in~\cite{Aussel2013-4} and also simplify some proofs.

The Minty Variational Inequality Problem (MVIP) was first studied by Debrunner \& Flor \cite{Debrunner1964}, 
and then by Minty~\cite{minty1967}.
However, the first connection between the MVIP and generalized monotonicity was due to John~\cite{John2001}.
We conclude this work  studying the relation between the MVIP asociated to a multivalued operator and his quasimonotone polar. 
We show a new characterization of quasimonotonicity in terms of the relation between the solutions of the
MVIP associated to the operator and its polar. Moreover, when the operator is pre-maximal quasimonotone we obtain that 
the solution set of the associated MVIP must be at most, a singleton.

\section{Preliminary definitions and notations}
Let $U,V$ be non-empty sets. 
A \emph{multivalued operator} $T:U\tos V$ is an application $T:U\to \mathcal{P}(V)$, that is, for $u\in U$, $T(u)\subset V$. The \emph{graph}, \emph{domain} and \emph{range} of $T$ are defined, respectively, as
\begin{gather*}
\dom(T)=\big\{u\in U\::\: T(u)\neq\emptyset\big\},\qquad \ran(T)=\bigcup_{u\in U}T(u),\\ 
\text{and }\gra(T)=\big\{(u,v)\in U\times V\::\: v\in T(u)\big\}.
\end{gather*}
From now on, we will identify multivalued operators with their graphs, so we will write $(u,v)\in T$ instead of $(u,v)\in \gra(T)$.

Let $V$ be a real vector space, and let $A\subset V$. The \emph{convex hull} of $A$, denoted as $\co(A)$, is the smallest convex set (in the sense of inclusion) which contains $A$. The \emph{conic hull} of $A$ is the set
\[
\cone(A)=\{tv\in V\::\:t\geq 0,\, v\in A\}.
\]
Furthermore, given $T:U\tos V$, we denote $\cone(T):U\tos V$ as the operator defined by $\cone(T)(x)=\cone(T(x))$, for all $x\in\dom(T)$. In addition, the \emph{effective domain} of $T$ is the set $\edom(T)=\{u\in U\::\:T(u)\neq \{0\}\}$. Note that $T\subset \cone(T)$ and $\edom(\cone(T))=\edom(T)=\edom(T\cup(X\times \{0\}))$.

Let $X$ be a Banach space and $X^*$ be its topological dual. 
The \emph{duality product} $\pi:X\times X^*\to\R$ is defined as $\pi(x,x^*)=\inner{x}{x^*}=x^*(x)$.
Let $C\subset X$ be a non-empty set. The \emph{normal cone} $N_C$ of $C$ at $x$ is defined as
\[
N_{C}(x)=\{x^*\in X^*\::\: \inner{y-x}{x^*}\leq 0,\,\forall\, y\in C\}.
\]
It is common in this definition to impose $C$ convex and $x\in C$, however, we won't consider such conditions in this work.

The \emph{monotone polar}~\cite{BSML} of an operator $T:X\tos X^*$ is the operator $T^{\mu}$ defined as
\[
T^{\mu}=\{(x,x^*)\in X\times X^*\::\: \inner{x-y}{x^*-y^*}\geq 0,\,\forall\,(y,y^*)\in T\}.
\]
The notions of \emph{monotonicity}, \emph{maximal} and \emph{pre-maximal} monotonicity can be expressed in term of the monotone polar. More precisely: 
\begin{enumerate}
\item $T$ is monotone if, and only if, $T\subset T^{\mu}$;
\item $T$ is maximal monotone if, and only if, $T=T^{\mu}$;
\item $T$ is pre-maximal monotone if, and only if, $T$ and $T^{\mu}$ are monotone.
\end{enumerate}

\section{The quasimonotone polar}
From now on $X$ will be a real Banach space and $X^*$ is its topological dual. Given $(x,x^*),(y,y^*)\in X\times X^*$, we say that $(x,x^*)$ is in a \emph{quasimonotone relation} with $(y,y^*)$, denoted by $(x,x^*)\sim_q(y,y^*)$, if 
\[
 \min\{\inner{y-x}{x^*},\inner{x-y}{y^*}\}\leq 0.
\]
Observe that $\sim_q$ is a reflexive and symmetric relation over $X\times X^*$.

\begin{definition}
The \emph{quasimonotone polar} of $T:X\tos X^*$, $T^{\nu}$ is given by
\[
 T^{\nu}=\{(x,x^*)\in X\times X^*\::\:(x,x)\sim_q(y,y^*),\,\forall\,(y,y^*)\in T\}.
\]
\end{definition}
Is easy to see that $T^{\mu}\subset T^{\nu}$, for every operator $T:X\tos X^*$.

Our notions of quasimonotone relation and polar are parallel to the monotone ones. They also were addressed by Aussel and Eberhard in~\cite{Aussel2013-4}. In particular, the quasimonotone polar $T^{\nu}$ was denoted as $QMR_T$ in~\cite{Aussel2013-4}.

\begin{example}\label{ejm:X0}
For $x\in X$, $\{(x,0)\}^{\nu}=X\times X^*$. Moreover, is straightforward to verify that $(X\times\{0\})^{\nu}=X\times X^*$.
\end{example}

Similarly to the monotone polar, the map $T\mapsto T^{\nu}$ is a \emph{polarity}~\cite{Birk79}. Thus, the following properties hold.
\begin{proposition}\label{pro:qpolar}\quad
\begin{enumerate}
\item $\ds\left(\bigcup_{i\in I}A_i\right)^{\nu}=\bigcap_{i\in I}A_i^{\nu}$.
\item $T\subset T^{\nu\nu}$.
\item $T^{\nu\nu\nu}=T^{\nu}$.
\item If $T\subset S$ then $S^{\nu}\subset T^{\nu}$.
\end{enumerate}
\end{proposition}

\begin{lemma}\label{lem:qcone}
Let $T:X\tos X^*$ be a multivalued operator and let $(x,x^*)\in T^{\nu}$, $(y,y^*)\in T$ and $t,s\geq 0$. Then $(x,tx^*)\sim_q(y,sy^*)$.
\end{lemma}
\begin{proof}
Note that the signs of $\inner{x-y}{sy^*}$ and $\inner{y-x}{tx^*}$ are the same as $\inner{x-y}{y^*}$ and $\inner{y-x}{x^*}$, respectively. Therefore, the sign of $\min\{\inner{x-y}{sy^*},\inner{y-x}{tx^*}\}$ also remains the same and the lemma follows.
\end{proof}

\newcounter{itemp}
\begin{proposition}\label{pro:qpolarpro} Let $T,S:X\tos X^*$ be multivalued operators. Then the following assertions hold.
\begin{enumerate}
\item $X\times\{0\}\subset T^{\nu}$.
\item If $S\subset X\times\{0\}$ then $(T\cup S)^{\nu}=T^{\nu}$. 
\item $\cone(T)^{\nu}=T^{\nu}=\cone(T^{\nu})$. In particular $T^{\nu}(x)$ is a cone, for all $x\in X$.
\item For all $x\in X$, $N_{\co\edom(T)}(x)\subset T^{\nu}(x)$.
\end{enumerate}
\end{proposition}

\begin{proof}
\begin{enumerate}
 \item By Example~\ref{ejm:X0}, given $x\in X$, $(x,0)$ is quasimonotonically related to any point in $X\times X^*$ and, 
 in particular, with any point in $T$. 
 \item Since $S\subset X\times\{0\}$, then $X\times X^*=(X\times\{0\})^{\nu}\subset S^{\nu}$. Therefore, using item {\it 1.} of Proposition~\ref{pro:qpolar}, $(T\cup S)^{\nu}=T^{\nu}\cap S^{\nu}=T^{\nu}\cap X\times X^*=T^{\nu}$.
 \item Since $T\subset\cone(T)$, $\cone(T)^{\nu}\subset T^{\nu}\subset\cone(T^{\nu})$. The converse inclusions follow from 
Lemma~\ref{lem:qcone}.
\item Denote $C=\co\edom(T)$. 
Let $x\in X$, take $x^*\in N_C(x)$ and pick any $(y,y^*)\in T$. If $y^*=0$ then, clearly, $(x,x^*)$ and $(y,y^*)$ are quasimonotonically related. If $y^*\neq 0$ then $y\in\edom(T)\subset C$ and, since $x^*\in N_C(x)$
\[
\inner{y-x}{x^*}\leq 0,
\]
that is, $(x,x^*)\sim_q(y,y^*)$. This implies $x^*\in T^{\nu}(x)$.\qedhere
\end{enumerate}
\end{proof}

Let $T:X\tos X^*$ be a multivalued operator. For $x\in X$, we define the set
\[
V_T(x)=\{y\in X\::\:\exists y^*\in T(y)\text{ such that }\inner{x-y}{y^*}>0\}.
\]

Clearly, given $T,S:X\tos X^*$ such that $T\subset S$, then $V_T(x)\subset V_S(x)$, for all $x\in X$
We now relate this set with the quasimonotone polar $T^{\nu}$.
\begin{lemma}\label{lem:VTempty}
	Let $T:X\tos X^*$. Then $V_T(x)=\emptyset$ if, and only if, $T^{\nu}(x)=X^*$.
\end{lemma}
\begin{proof}
	Assume that $V_T(x)=\emptyset$ and take any $x^*\in X^*$. Given $(y,y^*)\in T$, our assumption implies that $y\notin V_T(x)$, therefore $\inner{x-y}{y^*}\leq 0$. This in turn implies $(x,x^*)\sim_q(y,y^*)$, that is, $x^*\in T^{\nu}(x)$, since $(y,y^*)\in T$ was arbitrary.
	
	Reciprocally, suppose $V_T(x)\neq \emptyset$. Hence, there exists $(y,y^*)\in T$ such that $\inner{x-y}{y^*}>0$ and, therefore, $y\neq x$. A direct consequence of Hahn-Banach theorem allows us to obtain $x^*\in X^*$ such that $\inner{y-x}{x^*}>0$, hence $x^*\notin T^{\nu}(x)$. The lemma follows.
\end{proof}

\begin{proposition}\label{prop:Tnuequiv}
	Let $T:X\tos X^*$ and let $x\in X$ such that $V_T(x)\neq \emptyset$. Then $T^{\nu}(x)=N_{V_T(x)}(x)$.	
\end{proposition}
\begin{proof}
	Let $x^*\in N_{V_T(x)}(x)$ and $(y,y^*)\in T$. If $y\notin V_T(x)$ then $\inner{x-y}{y^*}\leq 0$ and $(x,x^*)\sim_q(y,y^*)$. On the other hand, if $y\in V_T(x)$, then $\inner{y-x}{x^*}\leq 0$ and, again,  $(x,x^*)\sim_q(y,y^*)$. Altogether, we conclude $x^*\in T^{\nu}(x)$.  
	
	Conversely, let $x^*\in T^{\nu}(x)$ and take an arbitrary $y\in V_T(x)$. Then, there exists $y^*\in T(y)$ such that $\inner{x-y}{y^*}>0$. This, along with $(x,x^*)\sim_q (y,y^*)$, implies $\inner{y-x}{x^*}\leq 0$, that is, $x^*\in N_{V_T(x)}(x)$.
\end{proof}

As a direct consequence from Proposition~\ref{prop:Tnuequiv}, we obtain the following corollary.
\begin{corollary}
Let $T:X\tos X^*$ be a multivalued operator and $x\in X$. Then $T^{\nu}(x)$ is a weak$^*$-closed, conic and convex set.
\end{corollary}

\begin{example}\label{ejm:ident}
	Let $H$ be a Hilbert space and consider $I:H\to H$, $I(x)=x$, the identity operator, which is maximal monotone.  
	For $x\in H$, $x\neq 0$, consider the set
	\[
	V_I(x)=\{y\in H\::\:\inner{x-y}{y}>0\}.
	\]
	It is straightforward to prove that $V_I(x)=\dfrac{1}{2}B(x,\|x\|)$ is the open ball with center $x/2$ and radius $\|x\|/2$. Thus, by Proposition~\ref{prop:Tnuequiv}, if $x\neq 0$,
	\[
	I^{\nu}(x)=N_{V_I(x)}(x)=\cone(\{x\}).
	\]
	On the other hand, $V_I(0)=\emptyset$ so $I^{\nu}(0)=H$, and we obtain, in conclusion,
	\[
	I^{\nu}=\cone(I)\cup (\{0\}\times H).
	\]
\end{example}

\section{On quasimonotone operators}

Recall that an operator $T:X\tos X^*$ is \emph{quasimonotone} if, for every $(x,x^*),(y,y^*)\in T$,
\[
\inner{y-x}{x^*}>0\quad\Longrightarrow\quad \inner{y-x}{y^*}\geq 0,
\]
or, equivalently, every $(x,x^*),(y,y^*)\in T$  are quasimonotonically related, that is,
\[
\min\{\inner{y-x}{x^*},\inner{x-y}{y^*}\}\leq 0.
\]

Follows directly from the definition that if $T\subset S$ and $S$ is quasimonotone, then $T$ must be quasimonotone as well.

\begin{proposition} Let $T:X\tos X^*$. The following conditions are equivalent:
\begin{enumerate}
 \item $T$ is quasimonotone,
 \item $T\subset T^{\nu}$,
 \item $T^{\nu\nu}\subset T^{\nu}$,
 \item $T^{\nu\nu}$ is quasimonotone.
 \item $\cone(T)$ is quasimonotone.
 \item $\cone(T)\cup(X\times\{0\})$ is quasimonotone.
\end{enumerate}
\end{proposition}
\begin{proof}
Equivalence between items {\it 1, 2, 3} and {\it 4} is a direct consequence of the definition of quasimonotone polar, 
and is analogous to~\cite[Proposition 21]{BSML}. On the other hand, since $T\subset\cone(T)\subset\cone(T)\cup(X\times\{0\})$ then {\it 6} implies {\it 5} which in turn implies {\it 1}. Finally, item {\it 2} implies $\cone(T)\subset\cone(T^{\nu})$, hence, using Proposition~\ref{pro:qpolarpro},
\[
\cone(T)\cup(X\times\{0\})\subset \cone(T^{\nu})\cup(X\times\{0\})=T^{\nu}=(\cone(T)\cup(X\times\{0\})^{\nu}.
\]
That is, item {\it 6}.
\end{proof}

The following corollary completes Lemma~4 in~\cite{Aussel2013-4}, and is analogous to~\cite[Proposition 19]{BSML}.
\begin{corollary}\label{123}
	Let $T:X\tos X^*$ be a quasimonotone operator. Then $(x,x^*)\in T^\nu$ if, and only if, $T\cup\{(x,x^*)\}$ is quasimonotone.
\end{corollary}

Denote $\widecheck T=\cone(T)\cup(X\times\{0\})$. Note that $\edom(T)=\edom(\widecheck T)$.

\begin{corollary} 
An operator $T:X\tos X^*$ is quasimonotone if, and only if, $\widecheck T \subset T^{\nu}$.
\end{corollary}

We say that $T$ is \emph{maximal quasimonotone} if $T$ is quasimonotone and it is maximal in the sense of inclusion, that is, 
if $T\subset S$ and $S$ is quasimonotone, then $T=S$.  A direct consequence of Zorn's Lemma shows that every quasimonotone 
operator possesses a maximal quasimonotone extension, that is, a maximal quasimonotone operator which contains it.

Given $T:X\tos X^*$, let $Q(T)$ be the set of maximal quasimonotone extensions of $T$, that is
\[
Q(T)=\{M: X\tos X^*\::\: M\text{ is maximal quasimonotone and }T\subset M\}.
\]
When $T$ is a quasimonotone operator, we can relate its quasimonotone polar to $Q(T)$. The following proposition is analogous to~\cite[Proposition 22]{BSML}. 
\begin{proposition}
Let $T:X\tos X^*$ be a quasimonotone operator. Then
\begin{enumerate}
 \item $T^{\nu}=\displaystyle\bigcup_{M\in Q(T)}M$;
 \item $T^{\nu\nu}=\displaystyle\bigcap_{M\in Q(T)}M$;
 \item $T$ is maximal quasimonotone if, and only if, $T=T^{\nu}$.
\end{enumerate}
\end{proposition}

\begin{corollary}
If $T:X\tos X^*$ is maximal quasimonotone, then $T(x)$ is a weak$^*$-closed, conic and convex set, for all $x\in X$.
\end{corollary}

\subsection{Lateral Translations}

Let $T:X\tos X^*$ be a multivalued operator and let $x_0\in X$. We consider the \emph{lateral translation} of $T$ by $x_0$, $\tau_{x_0}T:X\tos X^*$ defined as $\tau_{x_0}T=T+\{(x_0,0)\}$ or, equivalently, $\tau_{x_0}T(x)=T(x-x_0)$. We now show that this kind of translation will preserve quasimonotonicity and maximal quasimonotonicity.
\begin{proposition}\label{pro:lateral}
Let $T:X\tos X^*$ be a multivalued operator and let $x_0\in X$. Then $(\tau_{x_0}T)^{\nu}=\tau_{x_0}(T^{\nu})$.
Moreover, $T$ is quasimonotone (resp. maximal quasimonotone) if, and only if, $\tau_{x_0}T$ is quasimonotone (resp. maximal quasimonotone).
\end{proposition}
\begin{proof}
Let $(x,x^*)\in(\tau_{x_0}T)^{\nu}$ and pick any $(y,y^*)\in T$. Then $(y+x_0,y^*)\in \tau_{x_0}T$ and
\[
\min\{\inner{x-x_0-y}{y^*},\inner{y-(x-x_0)}{x^*}\}\leq 0.
\]
This implies $(x-x_0,x^*)\sim_q(y,y^*)$, for arbitrary $(y,y^*)\in T$, that is $(x,x^*)-(x_0,0)\in T^{\nu}$. Therefore $(x,x^*)\in\tau_{x_0}(T^{\nu})$. The reciprocal assertion is analogous. Invariance of both quasi and maximal quasimonotonicity is a direct consequence of the invariance of the quasimonotone polar by lateral translations.
\end{proof}

\subsection{On maximal quasimonotonicity}

In~\cite{Aussel2013-4}, Aussel and Eberhard defined another notion of maximality for quasimonotone operators, 
which from now on will be called \emph{AE-maximality}. Explicitly, $T: X\tos X^*$ is \emph{AE-maximal quasimonotone}, 
if $T$ is quasimonotone and, for every quasimonotone operator $S$ such that $T(x)\subset \cone(S(x))$, for all $x\in \edom(T)$, 
then $\cone(T(x))=\cone(S(x))$, for all $x\in\edom(T)$ and $\edom(T)=\edom(S)$. 

We now study the connection between these two notions of maximality. The following theorem recovers and extends Proposition~3 and~5 in \cite{Aussel2013-4}.

\begin{theorem}\label{teo:AEmaximal}
Let $T: X\tos X^*$ be a quasimonotone operator. The following conditions are equivalent:
\begin{enumerate}
 \item $T$ is AE-maximal quasimonotone;
 \item for every quasimonotone operator $S$ such that $\widecheck
 T\subset \widecheck S$, $\widecheck
 T=\widecheck S$;
 \item $\widecheck T$ is maximal quasimonotone;
 \item $\widecheck T=T^{\nu}$.
\end{enumerate}

\end{theorem}
\begin{proof}
First assume {\it 2}, and let $M$ be a quasimonotone operator such that $\widecheck T\subset M$. Then $\widecheck T\subset M\subset \widecheck M$. Using {\it 2}, we obtain that $\widecheck M = \widecheck T$, thus implying $\widecheck T= M$ and therefore, condition {\it 3}. Reciprocally, if $S$ is quasimonotone and $\widecheck T\subset \widecheck S$, maximal quasimonotonicity of $\widecheck T$ implies $\widecheck T=\widecheck S$. Thus, {\it 2} and {it 3} are equivalent. Moreover, note that Proposition~\ref{pro:qpolarpro}, item {\it 3.}, implies the equivalence between {\it 3} and {\it 4}.

Assume {\it 1.}, that is, $T$ is AE-maximal quasimonotone, and let $S$ be a quasimonotone operator such that $\widecheck T\subset \widecheck S$. Then, for every $x\in\edom(T)$, 
\[
\{0\}\neq T(x)\subset\cone(T)(x)=\widecheck T(x)\subset \widecheck S(x)=\cone(S)(x).
\]
By AE-maximality of $T$, $\cone(T)(x)=\cone(S)(x)$, for every $x\in\edom(T)$, and $\edom(T)=\edom(S)$. This implies that $\widecheck T=\widecheck S$.  On the other hand, assume that $T(x)\subset \cone(S)(x)$, for all $x\in\edom(T)$, for a certain quasimonotone operator $S$. Hence $\widecheck T\subset \widecheck S$, which, by {\it 2}, implies $\widecheck T=\widecheck S$. Therefore $\edom(T)=\edom(S)$ and $\cone(T)(x)=\cone(S)(x)$, for all $x\in\edom(T)$. This means that $T$ is AE-maximal quasimonotone.
\end{proof}

A direct corollary of Proposition~\ref{pro:lateral} and Theorem~\ref{teo:AEmaximal} is the following.
\begin{corollary}
Let $T:X\tos X^*$ be a AE-maximal quasimonotone operator and $x_0\in X$. Then $\tau_{x_0}T$ is also AE-maximal quasimonotone.
\end{corollary}

If $T$ is a quasimonotone operator such that $T^{\nu}$ is quasimonotone (therefore, maximal quasimonotone), 
then $T$ will be called \emph{pre-maximal quasimonotone}.  A pre-maximal quasimonotone operator is not 
necessarily maximal, but has a unique maximal quasimonotone extension.

\begin{corollary}
Let $T: X\tos X^*$ be a quasimonotone operator. If $T$ is AE-maximal quasimonotone then $T$ is pre-maximal quasimonotone. 
\end{corollary}
\begin{proof} 
Is enough to observe that $T$ AE-maximal implies that $T^{\nu}=\widecheck T$ is (maximal) quasimonotone.
\end{proof}

However, the converse assertion does not hold.

\begin{example}[Pre-maximality doesn't imply AE-maximality]\label{ex:prenotae}
	Let $T=\mathbb{Z}\times\{1\}$, which is monotone, therefore, quasimonotone. Now, given $(a,b)\in\R^2$ with $b>0$, note that $\{(a,b)\}^{\nu}=\R^2\setminus\{(x,y)\in\R^2\::\:x>a,\,y<0\}$. Therefore,
	\[
	T^{\nu}=\left(\bigcup_{a\in\mathbb{Z}}\{(a,1)\}\right)^{\nu}=\bigcap_{a\in\mathbb{Z}}\{(a,1)\}^{\nu}=\{(a,b)\in\R^2\::\:b\geq 0\},
	\]
	so $T^{\nu}$ is quasimonotone and $T$ is pre-maximal quasimonotone. However $T$ isn't AE-maximal quasimonotone, 
	since $(1/2,1)\in T^{\nu}\setminus \widecheck T$.
\end{example}

Pre-maximal quasimonotone operators share the same properties as their monotone counterparts. The following proposition is analogous to~\cite[Proposition 36]{BSML}.
\begin{proposition}
Let $T:X\tos X^*$ be a quasimonotone operator. Then the following conditions are equivalent:
\begin{enumerate}
 \item $T$ is pre-maximal quasimonotone;
 \item $T^\nu=T^{\nu\nu}$;
 \item $T^{\nu\nu}$ is maximal quasimonotone.
\end{enumerate}

\end{proposition}

For $T:X\tos X^*$ and $\alpha^*\in X^*$, we denote $T+\alpha^*:X\tos X^*$ to the operator defined as
\[
(T+\alpha^*)(x)=T(x)+\alpha^*,\qquad \forall x\in X.
\] 
The relation between the quasimonotonicity of the \emph{linear perturbations} $T+\alpha^*$ and the monotonicity of $T$ was established by Aussel, Corvellec and Lassonde~\cite{AusCorLass94}. Later on, Aussel and Eberhard extended this result to include maximality~\cite[Proposition~6]{Aussel2013-4}. We now present an analogous version of this result.
\begin{proposition}
Let $T:X\tos X^*$ such that $T+\alpha^*:X\tos X^*$ is pre-maximal quasimonotone, for all $\alpha^*\in X^*$. Then $T$ is pre-maximal monotone.
\end{proposition}
\begin{proof}
Monotonicity of $T$ is obtained from~\cite[Proposition 2.1]{AusCorLass94}, so remains to prove that $T^{\mu}$ is monotone. 
Let $(x,x^*), (y,y^*)\in T^\mu$. Then, for all  $\alpha^*\in X^*$,
\[
(x,x^*+\alpha^*), (y,y^*+\alpha^*)\in (T+\alpha^*)^\mu.
\]
In particular, for $\alpha_0^*=-\displaystyle\frac{x^*+y^*}{2}\in X^*$, 
\[
\left(x,\dfrac{1}{2}(x^*-y^*)\right), \left(y,\dfrac{1}{2}(y^*-x^*)\right)\in (T+\alpha_0^*)^\mu.
\]
Since $(T+\alpha^*)^\mu\subset (T+\alpha^*)^\nu$, and using the fact that $(T+\alpha^*)^\nu$ is quasimonotone, we obtain
\[
\frac{1}{2}\min\{\inner{y-x}{x^*-y^*},\inner{x-y}{y^*-x^*}\} \leq 0
\]
that is, $\langle x^*-y^*,x-y\rangle\geq0$. Therefore $T^\mu$ is monotone.
\end{proof}

\begin{remark}\label{rem:Talpha}
Proposition~6 in \cite{Aussel2013-4} states that if $T+\alpha^*$ is AE-maximal quasimonotone, 
for every $\alpha^*\in X^*$, then $T$ is maximal monotone.  
It is not difficult to present an example of a maximal monotone operator which is not AE-maximal quasimonotone. 
For instance, consider $T=I:\R\to\R$, $T(x)=x$, the identity operator. 
Moreover, no linear perturbation of the identity is AE-maximal quasimonotone.   
Even if an operator is AE-maximal quasimonotone, it could have a non AE-maximal quasimonotone linear perturbation, 
for example, $S:\R\tos\R$ defined as $S(x)=\{x/|x|\}$, if $x\neq 0$, and $S(0)=[-1,1]$, is maximal monotone and also 
AE-maximal quasimonotone, but $T+1$ is not. 
\end{remark}

\subsection{Further properties of quasimonotone operators}
Given $T:X\tos X^*$¨, denote by 
\[
E_T=\{x\in X\::\: V_T(x)=\emptyset\}.  
\]
Let $T,S:X\tos X^*$ be multivalued operators such that $T\subset S$. Since $V_T(x)\subset V_S(x)$, follows immediately that $E_S\subset E_T$. 

We now study some further properties of the set $E_T$. 
\begin{proposition}
Let $T:X\tos X^*$. Then $E_T$ is convex and weak$^*$-closed.
\end{proposition}
\begin{proof}
Note that 
\[
E_T=\bigcap_{(y,y^*)\in T}\{x\in X\::\:\inner{x-y}{y^*}\leq 0\}.
\]
The proposition follows since each of the sets in the intersection is convex and weak$^*$-closed.
\end{proof}

\begin{proposition}\label{pro:new1}
Let $T:X\tos X^*$ be a quasimonotone operator. Then $E_{T^{\nu}}$ is either empty or a singleton.
\end{proposition}
\begin{proof}
Assume that $E_{T^{\nu}}$ is non-empty and that there exist $u,v\in E_{T^{\nu}}$. In particular, $V_{T^{\nu}}(u)=\emptyset$ and, since $T$ is quasimonotone, $T\subset T^{\nu}$ and $V_T(u)=\emptyset$. Moreover, Lemma~\ref{lem:VTempty} now implies that $T^{\nu}(u)=X^*$.  On the other hand,  $u\notin V_{T^{\nu}}(v)=\emptyset$ and, therefore, $\inner{v-u}{u^*}\leq 0$, which in turn implies that $u=v$. Therefore $E_{T^{\nu}}$ is a singleton.
\end{proof}

\begin{proposition}\label{pro:new2}
Let $T:X\tos X^*$ be a pre-maximal quasimonotone operator. Then $E_T=E_{T^{\nu}}$.
\end{proposition}
\begin{proof}
Since $T$ is quasimonotone then $T\subset T^{\nu}$ and $E_{T^{\nu}}\subset E_T$.  Conversely, if $x\notin E_{T^{\nu}}$ then there exists $(y,y^*)\in T^{\nu}$ such that $\inner{x-y}{y^*}>0$ and, in particular, $x\neq y$. On the other hand, the Hahn-Banach theorem allows us to obtain $x^*\in X^*$ such that $\inner{y-x}{x^*}>0$. From the fact that $T^{\nu}$ is also quasimonotone, we conclude that $x^*\notin T^{\nu}(x)$, that is, $T^{\nu}(x)\neq X^*$ and, by by Lemma~\ref{lem:VTempty}, $x\notin E_T$. This proves the proposition.
\end{proof}

From the two previous propositions, we conclude this section with the following corollary.
\begin{corollary}\label{cor:new1}
If $T:X\tos X^*$ is pre-maximal quasimonotone, then $E_T$ is either empty or a singleton.
\end{corollary}


\section{On Minty Variational Inequality Problems}
  
The \emph{Minty Variational Inequality Problem} (MVIP) associated to a multivalued operator $T:X\tos X^*$ and a set $K\subset X$ is
\begin{equation}
\text{find $x\in K$ such that $\inner{x-y}{y^*}\leq 0$, for all $(y,y^*)\in T$, $y\in K$.}\tag{MVIP}
\end{equation}

The set of solutions of (MVIP) will be denoted as $M(T,K)$. Clearly, $M(T,K)$ is convex and closed 
provided that $K$ is convex and closed.
In addition, is immediate to observe the following proposition.
\begin{proposition}\label{prop:minteqet}
Let $T:X\tos X^*$ and $K\subset X$. Then 
\[
M(T,K)=\{x\in K\::\: V_T(x)\cap K=\emptyset\}.
\]
In particular, $M(T,X)=E_T$.
\end{proposition}

\begin{lemma}\label{inclu-Minty}
Let $T,S:X\tos X^*$ be multivalued operators and $K\subset X$. If $T\subset S$ then $M(S,K)\subset M(T,K)$.
\end{lemma}

\begin{lemma}\label{lem:new}
Let $R,S:X\tos X^*$ be multivalued operators such that $S\subset R^{\nu}$. If $M(S,K)\subset M(R,K)$, for all $K\subset X$, then $R$ is quasimonotone.
\end{lemma}
\begin{proof}
Assume that $R$ is not quasimonotone, so there exist $(u,u^*),(v,v^*)\in R$ such that 
\[
\inner{v-u}{u^*}>0\quad\text{and}\quad\inner{u-v}{v^*}>0.
\]
This implies that $u,v\notin M(R,\{u,v\})$. On the other hand, if $w^*\in S(u)\subset R^{\nu}(u)$ then, in particular, $(u,w^*)\sim_q(v,v^*)$, that is,
\[
\min\{\inner{u-v}{v^*},\inner{v-u}{w^*}\}\leq 0,
\]
which implies $\inner{v-u}{w^*}\leq 0$.  Since $\inner{v-v}{z^*}=0\leq 0$, for every $z^*\in S(v)$, we conclude that $v\in M(S,\{u,v\})\subset M(R,\{u,v\})$, a contradiction. The lemma follows.
\end{proof}

\begin{proposition}\label{prop:tnumint}
Let $T:X\tos X^*$ be an operator. Then $T$ is quasimonotone if, and only if, $M(T^\nu,K)\subset M(T,K)$, for every $K\subset X$. 
\end{proposition}
\begin{proof}
If $T$ is quasimonotone then $T\subset T^\nu$ and by, Lemma~\ref{inclu-Minty}, we obtain $M(T^\nu,K)\subset M(T,K)$.  The converse implication follows by taking $R=T$ and $S=T^{\nu}$ in Lemma~\ref{lem:new}.
\end{proof}

\begin{corollary}
Let $T:X\tos X^*$ be a multivalued operator. If $M(T^\nu,K)= M(T,K)$, for every $K\subset X$ then $T$ is pre-maximal quasimonotone.
\end{corollary}
\begin{proof}
By Proposition~\ref{prop:tnumint}, $T$ is quasimonotone.  Noting that $T\subset T^{\nu\nu}$ we can choose $R=T^{\nu}$ and $S=T$ in Lemma~\ref{lem:new}, to conclude that $T^{\nu}$ is quasimonotone as well. Therefore $T$ is pre-maximal monotone.
\end{proof}

The converse of the previous corollary is not true. Consider for example $T:\R\tos\R$ defined as
\[
T(x)=\begin{cases}\langle-\infty,0],&\text{if }x<0\\\R,&\text{if }x=0,\\\{0\},&\text{if }x=0.\end{cases}
\]
Is straightforward to prove that $T^{\nu}=\{(x,x^*)\in\R^2\::\:xx^*\geq 0\}$, which is quasimonotone. Therefore $T$ is pre-maximal quasimonotone. Let $K=[1,2]$. Then $M(T,K)=[1,2]$ and $M(T^{\nu},K)=\{1\}$.

We now use Proposition~\ref{prop:minteqet} to restate Propositions~\ref{pro:new1} and \ref{pro:new2}, and Corollary~\ref{cor:new1}, in our current setting. 
\begin{proposition}
Let $T:X\tos X^*$ be a quasimonotone operator. Then $M(T^{\nu},X)$ is either empty or a singleton. If, in addition, $T$ is pre-maximal quasimonotone, then $M(T,X)=M(T^{\nu},X)$ and $M(T,X)$ is either empty or a singleton.
\end{proposition}


\end{document}